\documentclass[12pt]{article}
\def\Dj{\hbox{D\kern-.73em\raise.30ex\hbox{-}
		\raise-.30ex\hbox{}}}
\def\dj{\hbox{d\kern-.33em\raise.80ex\hbox{-}
		\raise-.80ex\hbox{\kern-.40em}}}
\usepackage{epsfig}
\usepackage{amsmath,amsthm,amsfonts,amssymb,amscd,cite}
\allowdisplaybreaks

\newtheorem{theorem}{Theorem}[section]
\newtheorem{lemma}[theorem]{Lemma}
\newtheorem{corollary}[theorem]{Corollary}
\newtheorem{remark}[theorem]{Remark}
\newtheorem{definition}[theorem]{Definition}
\newtheorem{proposition}[theorem]{Proposition}

\begin{document}

	\baselineskip=0.30in

	\begin{center}
		{\Large \bf 
			On a novel graph associated with the circular space and its properties}
		
		\vspace{6mm}
		
		{\large \bf Sezer Sorgun $^a$\, Ali Gökhan Ertaş $^b$, İbrahim Günaltılı$^c$ }
		
		\vspace{9mm}
		
		\baselineskip=0.20in
		
		$^a${\it Department of Mathematics,\\
			Nevsehir Hac{\i} Bekta\c{s} Veli University, \\
			Nevsehir 50300, Türkiye.\/} \\
		$^b${\it Department of Informatics,\\
			Kütahya Dumlupınar  University, \\
			Kütahya, 43020, Türkiye.\/} \\
		$^c${\it Department of Mathematics,\\
			Eskişehir Osmangazi  University, \\
			Eskişehir, 26140, Türkiye.\/} \\
		e-mail: {\tt srgnrzs@gmail.com; aligokhanertas@gmail.com; igunalti@ogu.edu.tr}\\[3mm]

		\vspace{6mm}
		
		(Received 19 03, 2023)
		
	\end{center}
	
	\vspace{6mm}
	
	\baselineskip=0.23in

	\begin{abstract}
		
	A configuration of the triple $(\mathcal{P}, \mathcal{L}, \mathcal{I})$ on the incidence  relation which holds the properties of "Any two points are incident with at most one line" and "Any two lines are incident with at most one point". In projective geometry, bipartite graphs can be used as an incidence model between the points and the lines of a configuration. The graphs associated with a space are a good tool for understanding the topological and geometric properties of the space. in abstract systems. In this paper, we define a novel graph associated with circular space and obtain its properties in terms of some pure-graph invariants. Also, we characterize it regarding the graph associated with other spaces in the literature.  
		
		\bigskip
		
	\end{abstract}
	
	\vspace{5mm}
	
	\baselineskip=0.30in
	
	\section{Introduction}
	
	Let $G=(V,E)$ be a simple graph. If vertices $v_{i}$ and $v_{j}$ are adjacent, we denote that by $v_{i}v_{j}\in E(G)$ or $v_i\sim v_j$. The $distance$ between the vertices $u,v \in V(G)$, denoted by $d(u,v)$, is the minimum length of the paths between $u$ and $v$. The \textit{diameter} $diam(G)$ of $G$ is the maximum eccentricity among its vertices and the \textit{radius} $rad(G)$ is the minimum eccentricity of its vertices. Let $N_{v_i}$ be the neighbor set of a vertex $v_{i}$ in $ V(G)$. Throughout the paper, the common neighbour  of the vertices $v_1,v_2,\ldots v_k$ is denoted as $CN(v_1,\ldots, v_k)$ and $\vert CN(v_1,\ldots, v_k)\vert = cn(v_1,\ldots, v_k)$. A graph is a bipartite graph whose vertex set forms into two disjoint sets such that no two graph vertices within the same set are adjacent. 
	
	An incidence structure is a  $(\mathcal{P}, \mathcal{L}, \sigma)$ triple where $\mathcal{P}$ is a set whose elements are called points, $\mathcal{L}$ is a distinct set whose elements are called lines and $\sigma \subseteq \mathcal{P} \times \mathcal{L}$ is the incidence relation. For $(\mathcal{P}, \mathcal{L}, \sigma)$ triple, $G$ is a bipartite graph with the vertex set $V=\mathcal{P}\cup \mathcal{L}$ and edge set $E=\{\{p,l\}:p\sigma\mathcal{P}\}$ and it is first called incidence graph.  It is also known as   Levi graph \cite{levi1942finite}. In general, the Levi graph $G(\pi)$ of a plane $\pi$ is a bipartite incidence graph with $x,y$ forming an edge in the graph if and only if the point $x$ is on the line $y$. \\
	By motivation of Levi graphs, Hauschild et. al. \cite{hauschild2015levi} give a configuration of  $(\mathcal{P}, \mathcal{L}, \sigma)$ triple on the incidence  relation which holds the properties of "Any two points are incident with at most one line" and "Any two lines are incident with at most one point". In projective geometry, bipartite graphs can be used as incidence models between the points and the lines of a configuration. So the Levi graph of configuration $(\mathcal{P}, \mathcal{L}, \sigma)$ is the bipartite graph $G$ with $V(G)=\mathcal{P} \cup \mathcal{L}$ and $p \in \mathcal{P}$ is adjacent to $l\in \mathcal{L}$ if and only if $p \sigma l$ \cite{hauschild2015levi}. 
	
	 The neighborhood graph $\mathcal{N}(G)$ of a graph $G=(V, E)$ is a graph with vertex set $V\cup W$ where $W$ is the set of all open neighborhood sets of $G$ and with two vertices $u,w\in V\cup W$ adjacent if $u\in V$ and $w$ is in an open neighborhood set containing $u$ \cite{KulliVR2015}. Kulli \cite{KulliVR2015} gives some characterizations of $\mathcal{N}(G)$ for the extremal graph $G$. In \cite{konuralpjournalmath701085}, it has been defined the linear graph (also bipartite graph) with the help of linear spaces and obtained some results on the incidence graph (aka Levi graph) of the linear graphs. A linear graph is a bipartite graph with parts $\mathcal{P}$  and $\mathcal{L}$ satisfying two  conditions: " For all $p, q \in \mathcal{P}$ such that $p\neq q$, $cn(p,q)=1$" and $\delta(G)\ge2$.
	
	In this paper,  we obtained a novel bipartite graph whose name is circular graph associated the circular space given in  \cite{gunaltili2004finite,gunaltili2006finite}. We  give the properties of the circular graph with respect to graph invariants such as diameter, degree, etc. and characterize the circular graph regarding graphs associated with the other spaces.  

	\vspace*{3mm}
	
	\section{Main Results}
	\begin{definition} \cite{gunaltili2006finite} Let $P$ be a set of points, $C$ be a set of certain distinguished subsets of points called circles and $o\subseteq P \times C$. The incidence structure $C = (P, C, o)$ is called
		a circular space if:\\
		C1. Every circle contains at least three distinct points.\\
		C2. Any three distinct points are contained in exactly one circle.
	\end{definition}
	
	By the arguments in  Definition 2.1, we are ready for the novel graph whose name is a circular graph.
	
	\begin{definition} \label{A} Let $ G = (U\cup W,E) $ be any finite bipartite graph. $G$ is called a circular graph satisfying the following conditions
		\begin{itemize}
			\item[i.)]  $cn(u_i,u_j,u_k)=1$ for all $  u_i,u_j,u_k\in U $.
			\item[ii.)]  $d(w) \geq  3$, for all $w\in W $.
		\end{itemize}
		If $|U|=1$ or $|W|=1$ then $G$ is called a trivial circular graph. In this case, it is easy to see that $G \cong K_{1,n-1}$.
	\end{definition}
	
	\begin{figure}
		\centering
		\includegraphics[width=\textwidth]{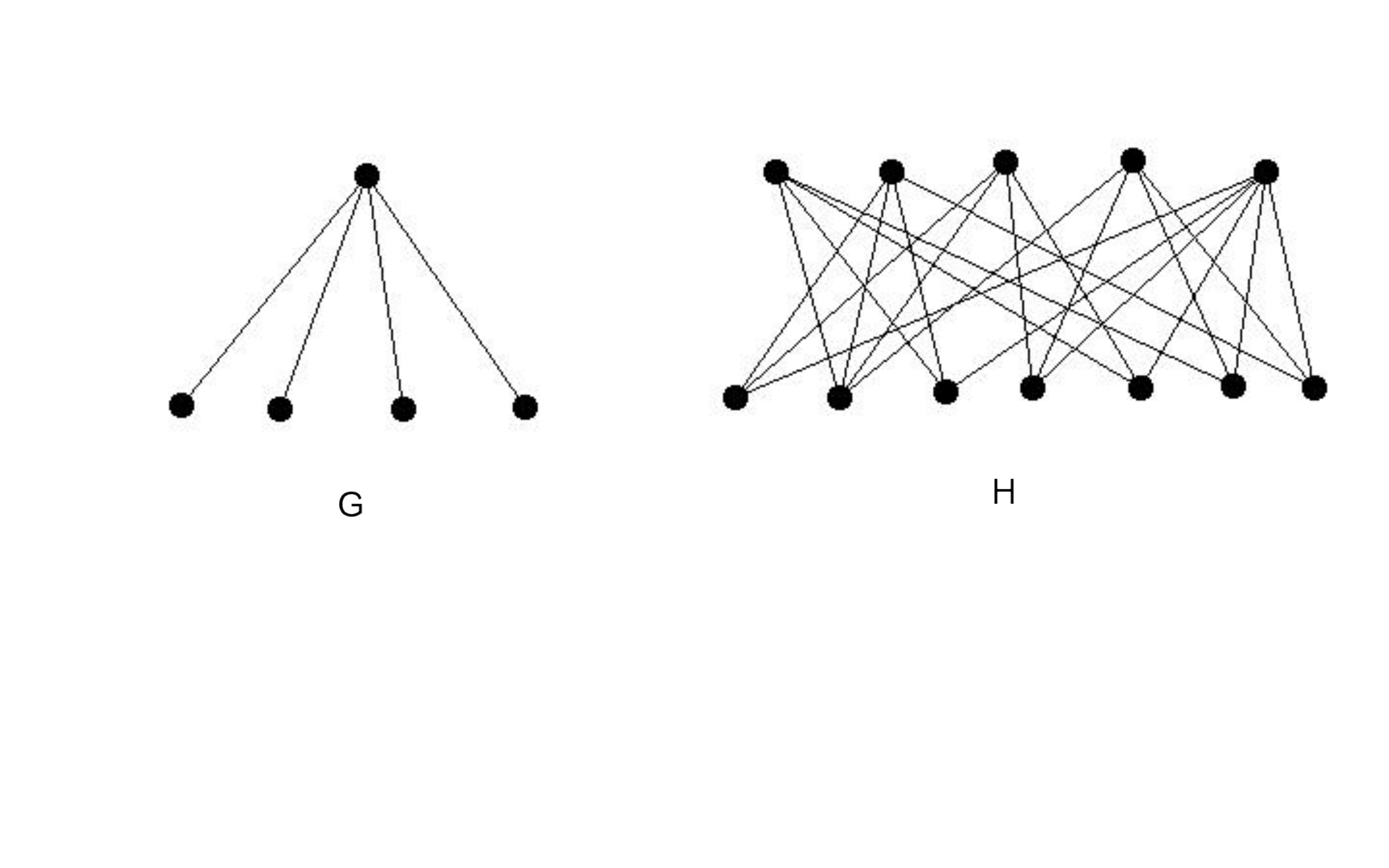}
		\caption{The samples of a trivial ($G$) and a non-trivial ($H$) circular graph}
	\end{figure}
	
	\begin{lemma} \label{B}
		Let $G = (U\cup W, E)$ be a circular graph. Then 
		\begin{equation*}
			cn(w_1,w_2)\leq 2
		\end{equation*}
		for every $w_1,w_2 \in  W$. 
	\end{lemma}
	
	\begin{proof}
		Let $G= (U\cup W, E)$ be a circular graph. Suppose that $\vert N(w_1)\cap N(w_2)\vert \geq 3$. Then there are  $(u_i,u_j,u_k)$ triple in the partition $U$ such that $\{u_i,u_j,u_k \} \subseteq CN(w_1, w_2)$. Hence, $cn(u_i,u_j,u_k) \geq 2$. But this is a contradiction because $cn(u_i,u_j,u_k )= 1$. Therefore we get $cn (w_1,w_2) \leq 2$.
	\end{proof}
	
	\begin{theorem}
		Let $\mathcal{T}$ be a family of trees. Then $T\in \mathcal{T}$ is a circular tree iff $T\cong K_{1,n-1}$.
	\end{theorem}
	
	\begin{proof}
		Let $T$ be any tree of order $n$ in $\mathcal{T}$. Since $T$  is also a bipartite graph, we consider as the vertex set $V=U\cup W$ where $U= \{u_1,\ldots, u_t\}$ and  $W= \{w_{1},\ldots,w_{s}\}$ such that $t+s=n$. Now consider $T$ as a circular tree. Then we have $cn(u_i,u_j,u_k)=1$ and  $d(w_i)\geq 3$. From Lemma \ref{B}, we also know that $cn(w_i,w_j ) \leq 2$  for $w_i,w_j\in W$. Assume that $cn(w_i,w_j)=2$. Then, we have $\{u_1,u_2,u_3\} \subseteq N(w_i)$ and  $\{u_2,u_3,u_4 \} \subseteq N(w_j)$  for $u_i\in U$ ($1\leq u_i\leq 4$). So we get $\{u_2,u_3 \} \subseteq CN(w_i,w_j)$ and hence  $T$  contains $w_i-u_2-w_j-u_3-w_i$ cycle but this is a contradiction because $T$ is tree. So $cn(w_1,w_2)\leq 1$.\\
		Now let $cn(w_i,w_j)=1$. Similarly, we have $\{u_3 \} \subseteq CN(w_i,w_j)$  such that $\{u_1,u_2,u_3 \} \subseteq N(w_i)$ and $\{u_3,u_4,u_5\} \subseteq N(w_j)$. In this case, $T$ also contains at least one cycle of length 4. But, the fact that T is a tree creates a contradiction. Hence, we get $cn(w_i,w_j )=0$. So, since any triple of $(u_i,u_j,u_k)$ has exactly one common neighbor in $W$ and  $cn(w_i,w_j )=0$, also $T$ cannot be disconnected graph, $W$ must be a singleton partition. That's $U=\{u_1,\ldots, u_{n-1}\}$ and $W=\{w_1\}$. Therefore $T\cong K_{1,n-1}$.
		    \vspace{3mm}
		Conversely, if $T\cong K_{1,n-1}$, it is clear that $T$ is a circular tree (trivial) from Definition \ref{A}.
	\end{proof}

	\begin{lemma} \label{C}
		Let $G = (U\cup W, E)$ be a non-trivial circular graph. Then for all $x\in U$, $d(x)\geq3$.  
	\end{lemma}
	
	\begin{proof}
		Let $G = (U\cup W, E)$ non-trivial circular graph. Then there is at least $x\in U$ for $y\in W$ such that $x\not \in N(y)$. From  Definition \ref{A} (i),  for $(u_i,u_j,u_k)$ triple in $U$ we get $\{u_i,u_j,u_k\} \subseteq N(y)$ because $d(y)\geq 3$. Since  $x \not \in N(y)$, so we get $x\neq u_i,u_j,u_k$. By Definition \ref{A} (ii.), there are $w_1,w_2,w_3$ vertices in $W$ such that $CN(x, u_i,u_j)=\{w_1\}$, $CN(x, u_i,u_k)=\{w_2\}$ and $CN(x, u_j,u_k)=\{w_3\}$.\\
		Here since  $y\not \in N(x)$, then $y \neq w_1,(w_2,w_3)$. If $w_1=w_2$ , then $cn (u_i,u_j,u_k) \geq 2$. But this is a contradiction to Definition \ref{A} (i.). Hence, $w_1 \neq w_2$. Similarly, we have $w_1 \neq w_3$ and $w_2  \neq w_3$. So, we get $d(x) \geq 3$ since $\{w_1,w_2,w_3\} \subseteq N(x)$.
	\end{proof}

	\begin{theorem}
		Let  $G = (U\cup W, E)$  be a non-trivial circular graph. Then the following holds 
		\item[i.)] 	$d(u_1,u_2)=2$ for any $u_1,u_2 \in U$. 
		\item[ii.)]	$d(w_1,w_2 )=2$  or $d(w_1,w_2 )=4$ for any $w_1,w_2 \in W $. 
		\item[iii.)] $d(u_1,w_1 )  \in  \{ 1,3 \} $ for any $u_1 \in U $ and  $ w_1 \in W$.
	\end{theorem}
	
	\begin{proof}
		\item[i.)] Since $G$ is bipartite graph, $u_1\not\sim u_2$ for $u_1,u_2\in U$. Then  we have $d(u_1,u_2 ) \neq 1$ and there is $w\in W$ such that $CN(u_1,u_2,u_3 )= \{ w \}$ for every $u_3 \in U$  since $ cn(u_1,u_2,u_3 )=1$ in Definition \ref{A}. Hence the path between $u_1$ and $u_2$ must be the path of $u_1-w-u_2$, so we get  $ d(u_1,u_2 )=2 $.
		\item[ii.)] Similarly, $d(w_1,w_2 ) \neq 1 $ for $w_1,w_2 \in W$ since  $G$ is a bipartite graph. Also by Lemma \ref{B}, we have  $cn(w_1,w_2)\leq 2$.\\
		Case 1:  Let  $cn(w_1,w_2 )=0$. \\
		Let  $w_1,w_2\in W$. Then we have $d(w_1 ) \geq 3$ and $d(w_2 ) \geq 3$. Also there is  $u_i\in U$ ($1\leq i\leq 6$) such that $\{u_1,u_2,u_3 \} \subseteq N(w_1)$  and $ \{ u_4,u_5,u_6 \} \subseteq N(w_2)$. Since   $cn(u_1,u_3,u_4 )=1$ and hence  there is $w_3\in W$ such that $CN(u_1,u_3,u_4 )= \{w_3\}$, $w_1-u_3-w_3-u_4-w_2$ path is one of the shortest paths between $w_1$ and $w_2$. So we get $d(w_1,w_2 )=4$.\\
		\vspace{2mm}
		Case2: Let $1\leq cn(w_1,w_2)\leq 2$. In this case, there is at least one vertex  $u_1$ in $U$ such that $\{u_1\}\subseteq CN(w_1,w_2 )$. Hence we get $d(w_1,w_2 )=2$.\\
		\vspace{2mm}
		\item[iii.)]  If $u_1 \in N(w_1)$ for  $u_1\in U$ and $w_1 \in W$, $d(u_1,w_1)=1$. If $u_1 \notin N(w_1)$, then there are the vertices $u_2,u_3,u_4$ in $U$  such that $u_2,u_3,u_4 \in N(w_1)$. Since  $cn(u_1,u_2,u_3 )=1$ , we have $w_2 \in W$ such that $CN(u_1,u_2,u_3 )= \{w_2 \} $. Hence the path of $u_1-w_2-u_3-w_1$ is one of the shortest path between the vertices $u_1$ and $w_1$. Therefore $d(u_1,w_1 )=3$.
	\end{proof}
	
	\begin{corollary}
		Let $G = (U\cup W, E)$ be a circular graph.
		\item[i.)]	If $G$ is trivial circular graph, then $diam(G)=2$ and $rad(G)=1$ 
		\item[ii.)]	If $G$ is non-trivial circular graph, then $3\leq diam(G)\leq 4$ and $rad(G)=3$ . 
	\end{corollary}
	
	\begin{proposition}
		Let $G=(U\cup W,E)$ be a circular graph. Then 
		$ \mathcal{N}(G)\cong G\cup G$
	\end{proposition}

\begin{proof}
	Let $G=(U\cup W,E)$ be a circular graph where $U= \{u_1,u_2,\ldots,u_i\}$ ve $W= \{w_1,w_2,\ldots,w_j\}$. Since the fact that the neighborhood graph is bipartite graph, $\mathcal{N}(G)$ is a bipartite graph of order $2(i+j)$ whose vertex partitions $U \cup W= \{u_1,\ldots,u_i,w_1, \ldots,w_j\}$ and  $N(U)\cup N(W)= \{N(u_1 ),\ldots,N(u_i ),N(w_1),\ldots,N(w_j)\}$. For any $x\in U$ and $y\in N(U)$, the vertices $x$ and $y$ are not adjacent in $\mathcal{N}$ since $N(U)\subseteq W$. Similarly, $x\not\sim y$ for   $x\in W$ and $y\in N(W)$. Thus  $\mathcal{N}(G)$ is a disconnected graph whose components are  $G_1=(U \cup N(W),E_1)$ and  $G_2= (W\cup N(U),E_2)$ where $E_1 = \{xy:x \in U$, $y \in N(W) \}$ and $E_2= \{ xy:x \in W$ , $y\in N(U)\}$.\\
	 Now, let $(u_m,u_n,u_t)$ be any triple for $u_m,u_n,u_t \in U$. Then we have  $CN(u_m,u_n,u_t)= \{w_j \}$ in $G$. By identifying  $w_j$ by $N(w_j)$, we get $CN(u_m,u_n,u_t)=\{N(w_j)\}$ in $G_1$ because $ \{u_m,u_n,u_t \}\subseteq N(w_j)$. Hence  we get the neighbourhood of any vertex in $G$ and $G_1$ are the same by taking $N(w_j)=w_j$. So $G_1\cong G$. Similarly, it is easy to see that $G_2\cong G$. Therefore
	$ \mathcal{N}(G)\cong G\cup G$.
\end{proof}
	
	\begin{theorem}
		Let $G = (U\cup W,E)$ be a non-trivial circular graph. $G'=(U'\cup W', E')$ is a linear graph with  $u'w' \in E'$ for $u'\in U'$, $w'\in W'$ where $U'=U-\{u\}$ and $W'=W-S$ such that $S=\{x:x\not\in N(u)\}$ for any vertex $u$ in $U$.
	\end{theorem}
	
	\begin{proof}
		Let $G = (U\cup W,E)$ be a non-trivial circular graph and let $(u,q,r)$ be any triple such that $ CN(u,q,r)=\{w\}$ for $u,q,r\ U$. and $w\in W$. Then there is only one $w'\in W'$ such that $N(w')= N(w)- \{u\}= \{q,r\}$. So we get  $CN(q,r)=\{ w'\}$ in $G'$. Hence, this fact holds the first condition of linear graph. Since each pair $(q,r)$ has exactly one common neighbor, $d_{u'}\geq 2$ for $u'\in U'$. On the other hand, we have $d_{w'}\geq 2$ for $w'\in N(u)$ if we  delete the vertices in $W$ which are not in the neighborhood of $u$. Hence we get $\delta(G')\geq 2$. This also holds the second condition in the definition of linear graph.   
		 
	\end{proof}

\begin{figure}
	\centering
	\includegraphics[width=\textwidth]{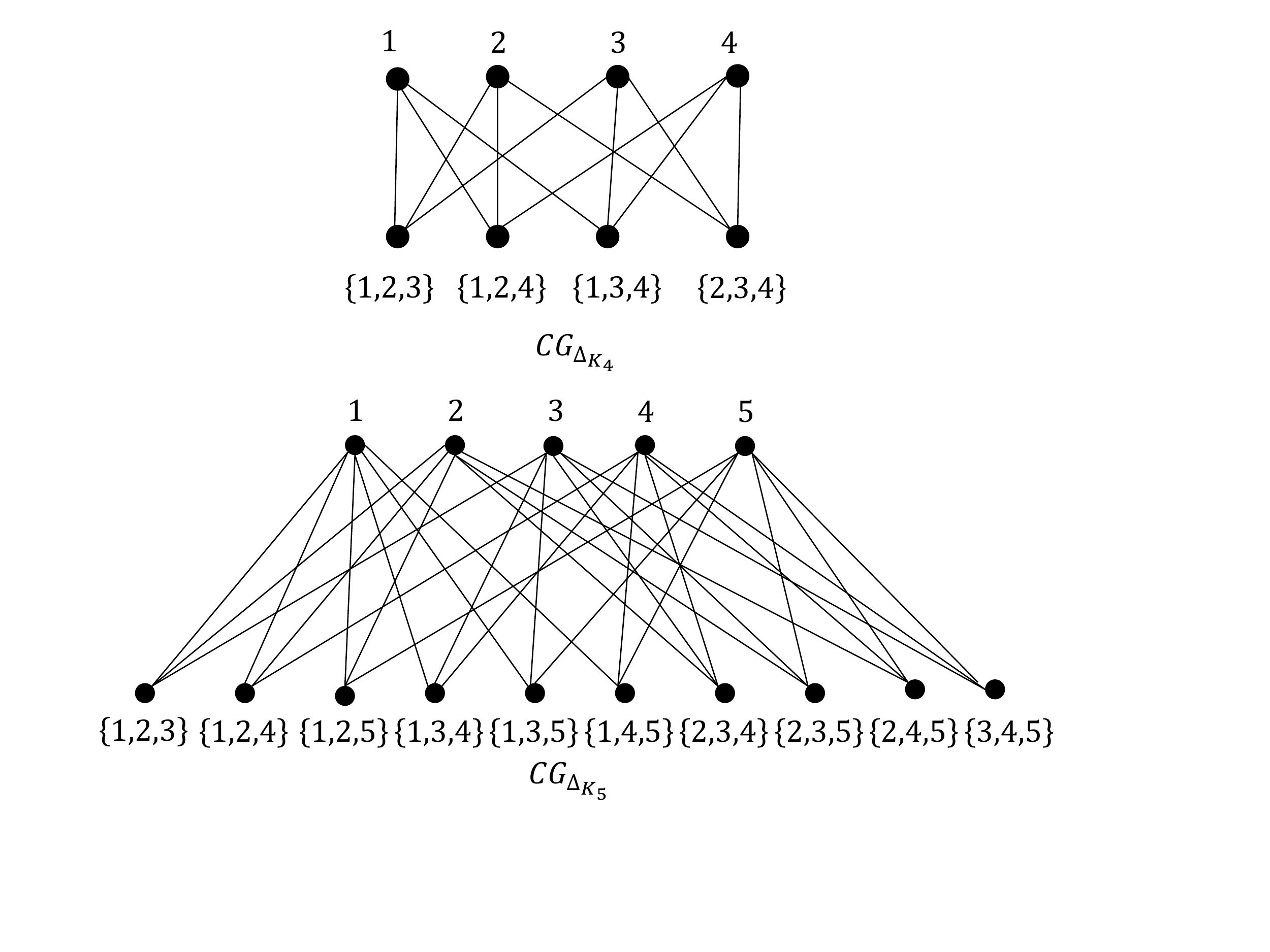}
	\caption{}
\end{figure}
	The following proposition gives a construction of circular graph (see Figure 2).
	\begin{proposition} \label{R}
		Let $K_n$ be a complete graph of order $n \geq 3$. Consider $U=V(K_n)$ and $W=\{ \{a,b,c\}: \text{different triangles in}\hspace{1mm} K_{n}\}$. Then $G = (U\cup W,E)$  is a circular graph with $uw\in E$ for $u\in U$, $w\in W$ and  $u\in N(w)$.  
	\end{proposition}

	\begin{proof}
		Let $K_n$ be a complete graph. Then there is $\binom{n}{3}$- different triangles. Then  any $u_i,u_j,u_k$ vertices in $V(K_n)$ form a triangle in $W$. Hence, there is only one $w$ in $W$ such that $CN(u_i,u_j,u_k )=\{w\}$. That's, this satisfies the condition of (i) in Definition \ref{A}. Also, $d(w)=3$  for $w \in W$.  Hence this also holds Definition \ref{A} (ii.).  
	\end{proof}

	\begin{remark}
		The circular graph, which is constructed in Proposition \ref{R}, can be defined as triangular circular graph and denoted by $CG_{\Delta_{K_n}}$. Also, $\mathcal{N}(K_4)\cong CG_{\Delta_{K_4}} $
	\end{remark}

\end{document}